\documentclass[twoside,reqno]{amsart}
\usepackage{amsfonts}
\usepackage{graphicx}
\usepackage{amscd}
\usepackage{cite}
\usepackage{hyperref}

\newtheorem{theorem}{Theorem}
\theoremstyle{plain}

\newtheorem{corollary}{Corollary}

\newtheorem{lemma}{Lemma}

\newtheorem{remark}{Remark}

\numberwithin{equation}{section}

\input{tcilatex}

\begin{document}
\title[]{Weak Convergence Theorem by a New Extragradient Method for Fixed
Point Problems and Variational Inequality Problems}
\author{\.{I}brahim Karahan}
\address{Department of Mathematics, Erzurum Technical University, Erzurum
25240, Turkey}
\email{ibrahimkarahan@erzurum.edu.tr}
\author{Murat \"{O}zdemir}
\address{Department of Mathematics, Ataturk University, Erzurum 25240, Turkey%
}
\email{mozdemir@atauni.edu.tr}

\begin{abstract}
We introduce a new extragradient iterative process, motivated and inspired \
by [S. H. Khan, A Picard-Mann Hybrid Iterative Process, Fixed Point Theory
and Applications, doi:10.1186/1687-1812-2013-69], for finding a common
element of the set of fixed points of a nonexpansive mapping and the set of
solutions of a variational inequality for an inverse strongly monotone
mapping in a Hilbert space. Using this process, we prove a weak convergence
theorem for the class of nonexpansive mappings in Hilbert spaces. Finally,
as an application, we give some theorems by using resolvent operator and
strictly pseudocontractive mapping.
\end{abstract}

\maketitle

\section{Introduction}

Let $H$ be a real Hilbert space with the inner product $\left\langle \cdot
,\cdot \right\rangle $ and the norm $\left\Vert \cdot \right\Vert $,
respectively. Let $C$ be a nonempty closed convex subset of $H,$ $I$ be the
idendity mapping on $C,$ and $P_{C}$ be the metric projection from $H$ onto $%
C.$

Recall that a mapping $T:C\rightarrow C$ is called nonexpansive if 
\begin{equation*}
\left\Vert Tx-Ty\right\Vert \leq \left\Vert x-y\right\Vert ,\text{ }\forall
x,y\in C.
\end{equation*}%
We denote by $F(T)$ the set of fixed points of $T$, i.e., $F\left( T\right)
=\left\{ x\in C:Tx=x\right\} $. For a mapping $A:C\rightarrow H,$ it is
called monotone if%
\begin{equation*}
\left\langle Ax-Ay,x-y\right\rangle \geq 0,
\end{equation*}%
$L$-Lipschitzian if there exists a constant $L>0$ such that 
\begin{equation*}
\left\Vert Ax-Ay\right\Vert \leq L\left\Vert x-y\right\Vert ,\text{ }\forall
x,y\in C;
\end{equation*}%
and $\alpha $-inverse strongly monotone if%
\begin{equation*}
\left\langle Ax-Ay,x-y\right\rangle \geq \alpha \left\Vert Ax-Ay\right\Vert
^{2},
\end{equation*}%
for all $x,y\in C.$

\begin{remark}
It is obvious that any $\alpha $-inverse strongly monotone mapping $A$ is
monotone and $\frac{1}{\alpha }$-Lipschitz continuous.
\end{remark}

Monotonicity conditions in the context of variational methods for nonlinear
operator equations were used by Vainberg and Kacurovskii \cite{vaka} and
then many authors have studied on this subject.

In this paper, we consider the following variational inequality problem $%
VI\left( C,A\right) $: find a $x\in C$ such that%
\begin{equation*}
\left\langle Ax,y-x\right\rangle \geq 0,\text{ \ }\forall y\in C.
\end{equation*}%
The set of solutions of $VI\left( C,A\right) $ is denoted by $\Omega ,$ i.e.,%
\begin{equation*}
\Omega =\left\{ x\in C:\left\langle Ax,y-x\right\rangle \geq 0\text{, }%
\forall y\in C\right\} .
\end{equation*}%
Let $S:C\rightarrow H$ be a mapping. In the context of the variational
inequality problem it is easy to check that%
\begin{equation*}
x\in \Omega \Leftrightarrow x\in F\left( P_{C}\left( I-\lambda S\right)
\right) ,\text{ }\forall \lambda >0.
\end{equation*}%
Variational inequalities were initially studied by Stampacchia \cite{stam1}, 
\cite{stam2}. Such a problem is connected with convex minimization problem,
the complementarity problem, the problem of finding point $x\in C$
satisfying $0\in A$ and etc. Fixed point problems are also closely related
to the variational inequality problems.

For finding an element of $F=F\left( T\right) \cap \Omega ,$ many authors
have studied widely under suitable assumptions (see \cite{sakasa, yalis,
caku, lita2, wosaya, yaliya}). For example, in 2006, Takahashi and Toyoda 
\cite{tato} introduced following iterative process:%
\begin{equation}
\left\{ 
\begin{array}{l}
x_{0}\in C \\ 
x_{n+1}=\alpha _{n}x_{n}+\left( 1-\alpha _{n}\right) TP_{C}\left( I-\lambda
_{n}A\right) x_{n},\text{ }\forall n\geq 0\text{,}%
\end{array}%
\right.   \label{A}
\end{equation}%
where $C$ is a nonempty closed convex subset of a real Hilbert space $H,$ $%
A:C\rightarrow H$ is an $\alpha $-inverse strongly monotone mapping, $%
P_{C}:H\rightarrow C$ is a metric projection, $T:C\rightarrow C$ is a
nonexpansive mapping, $\left\{ \alpha _{n}\right\} \subset \left[ a,b\right] 
$ for some $a,b\in \left( 0,1\right) ,$ and $\left\{ \lambda _{n}\right\}
\subset \left[ c,d\right] $ for some $c,d\in \left( 0,2\alpha \right) .$
They proved that if $F=F\left( T\right) \cap \Omega $ is nonempty, then the
sequence $\left\{ x_{n}\right\} $ generated by (\ref{A}) converges weakly to
some $z\in F$ where $z=\lim_{n\rightarrow \infty }P_{F}x_{n}.$ In the same
year, Nadezkhina and Takahashi \cite{nata} generalized the iterative process
(\ref{A}) and motivated by this process they introduced following iterative
scheme for nonexpansive mapping $S$ and monotone $k$-Lipschitzian mapping $A$%
. They proved the weak convergence of $\left\{ x_{n}\right\} $ under the
suitable conditions: 
\begin{equation}
\left\{ 
\begin{array}{l}
x_{0}\in C \\ 
x_{n+1}=\alpha _{n}x_{n}+\left( 1-\alpha _{n}\right) SP_{C}\left(
x_{n}-\lambda _{n}y_{n}\right)  \\ 
y_{n}=P_{C}\left( I-\lambda _{n}A\right) x_{n},\text{ }\forall n\geq 0.%
\end{array}%
\right.   \label{B}
\end{equation}

Recently, independetly from the above processes, Khan \cite{khan} and Sahu 
\cite{sahu}, individually, introduced the following iterative process which
Khan referred to as Picard-Mann hybrid iterative process:%
\begin{equation}
\left\{ 
\begin{array}{l}
x_{0}\in C \\ 
x_{n+1}=Ty_{n} \\ 
y_{n}=\alpha _{n}x_{n}+\left( 1-\alpha _{n}\right) Tx_{n},\text{ }\forall
n\geq 0,%
\end{array}%
\right.   \label{C}
\end{equation}%
where $\left\{ \alpha _{n}\right\} $ is a sequence in $\left( 0,1\right) .$
Khan proved a strong and a weak convergence theorems in Banach space for
iterative process (\ref{C}) under the suitable conditions where $T$ is a
nonexpansive mapping. Also he proved that the iterative process given by (%
\ref{C}) converges faster than the Picard, Mann and Ishikawa processes for
the contraction mappings.

In this paper, motivated and inspired by the idea of extragradient method
and the above processes, we introduce the following process:%
\begin{equation}
\left\{ 
\begin{array}{l}
x_{0}\in C \\ 
x_{n+1}=TP_{C}\left( I-\lambda _{n}A\right) y_{n} \\ 
y_{n}=\alpha _{n}x_{n}+\left( 1-\alpha _{n}\right) TP_{C}\left( I-\lambda
_{n}A\right) x_{n},\text{ }\forall n\geq 0,%
\end{array}%
\right.   \label{Z}
\end{equation}%
where $T$ is a nonexpansive mapping and $P_{C}$ is a metric projection from $%
H$ onto $C.$ Our iterative process is independent from all of the above
processes. Also, under the suitable conditions, we establish a weak
convergence theorem.

\section{Preliminaries}

In this section, we collect some useful lemmas that will be used for our
main result in the next section. We write $x_{n}\rightharpoonup x$ to
indicate that the sequence $\left\{ x_{n}\right\} $ converges weakly to $x,$
and $x_{n}\rightarrow x$ for the strong convergence. It is well known that
for any $x\in H,$ there exists a unique point $y_{0}\in C$ such that%
\begin{equation*}
\left\Vert x-y_{0}\right\Vert =\inf \left\{ \left\Vert x-y\right\Vert :y\in
C\right\} .
\end{equation*}%
We denote $y_{0}$ by $P_{C}x,$ where $P_{C}$ is called the metric projection
of $H$ onto $C.$ We know that $P_{C}$ is a nonexpansive mapping. It is also
known that $P_{C}$ has the following properties:

\begin{enumerate}
\item[(i)] $\left\Vert P_{C}x-P_{C}y\right\Vert \leq \left\Vert
x-y\right\Vert ,$ for all $x,y\in H,$

\item[(ii)] $\left\Vert x-y\right\Vert ^{2}\geq \left\Vert
x-P_{C}x\right\Vert ^{2}+\left\Vert y-P_{C}x\right\Vert ^{2},$ for all $x\in
H,$ $y\in C,$

\item[(iii)] $\left\langle x-P_{C}x,y-P_{C}x\right\rangle \leq 0,$ for all $%
x\in H,$ $y\in C,$
\end{enumerate}

It is known that a Hilbert space $H$ satisfies the Opial condition that, for
any sequence $\left\{ x_{n}\right\} $ with $x_{n}\rightharpoonup x,$ the
inequality%
\begin{equation*}
\liminf_{n\rightarrow \infty }\left\Vert x_{n}-x\right\Vert
<\liminf_{n\rightarrow \infty }\left\Vert x_{n}-y\right\Vert 
\end{equation*}%
holds for every $y\in H$ with $y\neq x.$

\begin{lemma}
\label{c}\cite{tato} Let $C$ be a nonempty closed convex subset of a real
Hilbert space $H$ and $\left\{ x_{n}\right\} $ be a sequence in $H.$ Suppose
that, for all $z\in C,$%
\begin{equation*}
\left\Vert x_{n+1}-z\right\Vert \leq \left\Vert x_{n}-z\right\Vert
\end{equation*}%
for every $n=0,1,2,\ldots .$ Then, $\left\{ P_{C}x_{n}\right\} $ converges
strongly to some $u\in C.$
\end{lemma}

\begin{lemma}
\cite{tato} Let $C$ be a nonempty closed convex subset of a real Hilbert
space $H$ and let $A$ be an $\alpha $-inverse strongly monotone mapping of $%
C $ into $H.$ Then the solution of $VI\left( C,A\right) $, $\Omega ,$ is
nonempty.
\end{lemma}

For a set-valued mapping $S:H\rightarrow 2^{H}$, if the inequality%
\begin{equation*}
\left\langle f-g,u-v\right\rangle \geq 0
\end{equation*}%
holds for all $u,v\in C,f\in Su,g\in Sv,$ then $S$ is called monotone
mapping. A monotone mapping $S:H\rightarrow 2^{H}$ is maximal if the graph $%
G\left( S\right) $ of $S$ is not properly contained in the graph of any
other monotone mappings. It is known that a monotone mapping $S$ is maximal
if and only if, for $\left( u,f\right) \in H\times H,$ $\left\langle
u-v,f-w\right\rangle \geq 0$ for every $\left( v,w\right) \in G\left(
S\right) $ implies $f\in Su.$ Let $A$ be an inverse strongly monotone
mapping of $C$ into $H,$ let $N_{C}v$ be the normal cone to $C$ at $v\in C,$
i.e.,%
\begin{equation*}
N_{C}v=\left\{ w\in H:\left\langle v-u,w\right\rangle \geq 0,\forall u\in
C\right\} ,
\end{equation*}%
and define%
\begin{equation*}
Sv=\left\{ 
\begin{array}{cc}
Av+N_{C}v & v\in C \\ 
\emptyset & v\notin C.%
\end{array}%
\right.
\end{equation*}%
Then $S$ is maximal monotone and $0\in Sv$ if and only if $v\in \Omega .$

\begin{lemma}
\label{b}\cite{kirk} Let $C$ be a nonempty closed convex subset of a real
Hilbert space $H,$ and $T$ be a nonexpansive self-mapping of $C.$ If $%
F\left( T\right) \neq \emptyset ,$ then $I-T$ is demiclosed; that is
whenever $\left\{ x_{n}\right\} $ is a sequence in $C$ weakly converging to
some $x\in C$ and the sequence $\left\{ \left( I-T\right) x_{n}\right\} $
strongly converges to some $y$, it follows that $\left( I-T\right) x=y.$
Here $I$ is the identity operator of $H.$
\end{lemma}

\begin{lemma}
\label{a}\cite{schu} Let $H$ be a real Hilbert space, let $\left\{ \alpha
_{n}\right\} $ be a sequence of real numbers such that $0<a\leq \alpha
_{n}\leq b<1$ for all $n=0,1,2,\ldots ,$ and let $\left\{ x_{n}\right\} $
and $\left\{ y_{n}\right\} $ be sequences of $H$ such that%
\begin{equation*}
\limsup_{n\rightarrow \infty }\left\Vert x_{n}\right\Vert \leq c,\text{ }%
\limsup_{n\rightarrow \infty }\left\Vert y_{n}\right\Vert \leq c\text{ and }%
\lim_{n\rightarrow \infty }\left\Vert \alpha _{n}x_{n}+\left( 1-\alpha
_{n}\right) y_{n}\right\Vert =c,
\end{equation*}%
\ for some $c>0.$ Then,%
\begin{equation*}
\lim_{n\rightarrow \infty }\left\Vert x_{n}-y_{n}\right\Vert =0.
\end{equation*}
\end{lemma}

\section{Main result}

In this section, we introduced a new extragradient method and proved that
the sequence generated by this iteration method converges weakly to a fixed
point of nonexpansive mapping and to a solution of variational inequality $%
VI(C,A)$.

\begin{theorem}
\label{11}Let $C$ be a nonempty closed convex subset of a real Hilbert space 
$H.$ Let $A:C\rightarrow H$ be an $\alpha $-inverse strongly monotone
mapping and $T:C\rightarrow C$ be a nonexpansive mapping such that $%
F=F\left( T\right) \cap \Omega \neq \emptyset .$ For arbitrary initial value 
$x_{0}\in H,$ let $\left\{ x_{n}\right\} $ be a sequence defined by (\ref{Z}%
) where $\{\lambda _{n}\}\subset \lbrack a,b]$ for some $a,b\in (0,2\alpha )$
and $\left\{ \alpha _{n}\right\} \subset \left[ c,d\right] $ for some $%
c,d\in \left( 0,1\right) $. Then, the sequence $\left\{ x_{n}\right\} $
converges weakly to a point $p\in F,$ where $p=\lim_{n\rightarrow \infty
}P_{F}x_{n}.$
\end{theorem}

\begin{proof}
We devide our proof into four steps.

\textbf{Step 1. }Let $t_{n}=P_{C}\left( I-\lambda _{n}A\right) x_{n}.$
First, we show that $\left\{ x_{n}\right\} $ and $\left\{ t_{n}\right\} $
are bounded sequences. Let $z\in F\left( T\right) \cap \Omega ,$ then, we
have that;%
\begin{eqnarray}
\left\Vert t_{n}-z\right\Vert ^{2} &=&\left\Vert P_{C}\left( I-\lambda
_{n}A\right) x_{n}-z\right\Vert ^{2}  \notag \\
&\leq &\left\Vert \left( I-\lambda _{n}A\right) x_{n}-\left( I-\lambda
_{n}A\right) z\right\Vert ^{2}  \notag \\
&=&\left\Vert x_{n}-z-\lambda _{n}\left( Ax_{n}-Az\right) \right\Vert ^{2} 
\notag \\
&\leq &\left\Vert x_{n}-z\right\Vert ^{2}-2\lambda _{n}\left\langle
x_{n}-z,Ax_{n}-Az\right\rangle +\lambda _{n}^{2}\left\Vert
Ax_{n}-Az\right\Vert ^{2}  \notag \\
&\leq &\left\Vert x_{n}-z\right\Vert ^{2}+\lambda _{n}\left( \lambda
_{n}-2\alpha \right) \left\Vert Ax_{n}-Az\right\Vert ^{2}  \notag \\
&\leq &\left\Vert x_{n}-z\right\Vert ^{2}  \label{1}
\end{eqnarray}%
and from (\ref{1}) we get%
\begin{eqnarray*}
\left\Vert x_{n+1}-z\right\Vert ^{2} &=&\left\Vert TP_{C}\left( I-\lambda
_{n}A\right) y_{n}-z\right\Vert ^{2} \\
&=&\left\Vert TP_{C}\left( I-\lambda _{n}A\right) y_{n}-TP_{C}\left(
I-\lambda _{n}A\right) z\right\Vert ^{2} \\
&\leq &\left\Vert y_{n}-z\right\Vert ^{2} \\
&=&\left\Vert \alpha _{n}\left( x_{n}-z\right) +\left( 1-\alpha _{n}\right)
\left( Tt_{n}-z\right) \right\Vert ^{2} \\
&\leq &\alpha _{n}\left\Vert x_{n}-z\right\Vert ^{2}+\left( 1-\alpha
_{n}\right) \left\Vert Tt_{n}-z\right\Vert ^{2} \\
&\leq &\alpha _{n}\left\Vert x_{n}-z\right\Vert ^{2}+\left( 1-\alpha
_{n}\right) \left\Vert t_{n}-z\right\Vert ^{2} \\
&\leq &\alpha _{n}\left\Vert x_{n}-z\right\Vert ^{2} \\
&&+\left( 1-\alpha _{n}\right) \left[ \left\Vert x_{n}-z\right\Vert
^{2}+\lambda _{n}\left( \lambda _{n}-2\alpha \right) \left\Vert
Ax_{n}-Az\right\Vert ^{2}\right] \\
&=&\left\Vert x_{n}-z\right\Vert ^{2}+\left( 1-\alpha _{n}\right) \lambda
_{n}\left( \lambda _{n}-2\alpha \right) \left\Vert Ax_{n}-Az\right\Vert ^{2}
\\
&\leq &\left\Vert x_{n}-z\right\Vert ^{2}+\left( 1-d\right) a\left(
b-2\alpha \right) \left\Vert Ax_{n}-Az\right\Vert ^{2} \\
&\leq &\left\Vert x_{n}-z\right\Vert ^{2}.
\end{eqnarray*}%
Therefore, there exists $\lim_{n\rightarrow \infty }\left\Vert
x_{n}-z\right\Vert $ and $Ax_{n}-Az\rightarrow 0.$ Hence $\left\{
x_{n}\right\} $ and $\left\{ t_{n}\right\} $ are bounded.

\textbf{Step 2.} We will show that $\lim_{n\rightarrow \infty }\left\Vert
x_{n}-y_{n}\right\Vert =0.$ Before that, we shall show $\lim_{n\rightarrow
\infty }\left\Vert Tt_{n}-x_{n}\right\Vert =0$. From Step 1, we know that $%
\lim_{n\rightarrow \infty }\left\Vert x_{n}-z\right\Vert $ exists for all $%
z\in F\left( T\right) \cap \Omega $.\ Let $\lim_{n\rightarrow \infty
}\left\Vert x_{n}-z\right\Vert =c.$\ Since%
\begin{equation*}
\left\Vert x_{n+1}-z\right\Vert \leq \left\Vert y_{n}-z\right\Vert \leq
\left\Vert x_{n}-z\right\Vert ,
\end{equation*}%
we get%
\begin{equation}
\lim_{n\rightarrow \infty }\left\Vert y_{n}-z\right\Vert =c.  \label{*1}
\end{equation}%
On the other hand, since%
\begin{equation*}
\left\Vert Tt_{n}-z\right\Vert \leq \left\Vert t_{n}-z\right\Vert \leq
\left\Vert x_{n}-z\right\Vert ,
\end{equation*}%
we have%
\begin{equation}
\limsup_{n\rightarrow \infty }\left\Vert Tt_{n}-z\right\Vert \leq c.
\label{*2}
\end{equation}%
Also, we know that%
\begin{equation}
\limsup_{n\rightarrow \infty }\left\Vert x_{n}-z\right\Vert \leq c
\label{*3}
\end{equation}%
and%
\begin{equation}
\lim_{n\rightarrow \infty }\left\Vert y_{n}-z\right\Vert =\lim_{n\rightarrow
\infty }\left\Vert \alpha _{n}\left( x_{n}-z\right) +\left( 1-\alpha
_{n}\right) \left( Tt_{n}-z\right) \right\Vert =c.  \label{*4}
\end{equation}%
Hence, from (\ref{*2}), (\ref{*3}), (\ref{*4}), and Lemma \ref{a} , we get
that%
\begin{equation}
\lim_{n\rightarrow \infty }\left\Vert x_{n}-Tt_{n}\right\Vert =0.  \label{*5}
\end{equation}%
We have also%
\begin{eqnarray*}
\left\Vert x_{n}-y_{n}\right\Vert &=&\left\Vert x_{n}-\alpha
_{n}x_{n}-\left( 1-\alpha _{n}\right) Tt_{n}\right\Vert \\
&=&\left( 1-\alpha _{n}\right) \left\Vert x_{n}-Tt_{n}\right\Vert .
\end{eqnarray*}%
So, from (\ref{*5}) we obtain that%
\begin{equation}
\lim_{n\rightarrow \infty }\left\Vert x_{n}-y_{n}\right\Vert =0.  \label{5.5}
\end{equation}%
Since $A$ is Lipschitz continuous, we have $Ax_{n}-Ay_{n}\rightarrow 0.$

\textbf{Step 3.} Next, we show that $\lim_{n\rightarrow \infty }\left\Vert
Tx_{n}-x_{n}\right\Vert =0.$ Using the properties of metric projections,
since%
\begin{eqnarray*}
\left\Vert t_{n}-z\right\Vert ^{2} &=&\left\Vert P_{C}\left( I-\lambda
_{n}A\right) x_{n}-P_{C}\left( I-\lambda _{n}A\right) z\right\Vert ^{2} \\
&\leq &\left\langle t_{n}-z,\left( I-\lambda _{n}A\right) x_{n}-\left(
I-\lambda _{n}A\right) z\right\rangle \\
&=&\frac{1}{2}\left[ \left\Vert t_{n}-z\right\Vert ^{2}+\left\Vert \left(
I-\lambda _{n}A\right) x_{n}-\left( I-\lambda _{n}A\right) z\right\Vert
^{2}\right. \\
&&\left. -\left\Vert t_{n}-z-\left[ \left( I-\lambda _{n}A\right)
x_{n}-\left( I-\lambda _{n}A\right) z\right] \right\Vert ^{2}\right] \\
&\leq &\frac{1}{2}\left[ \left\Vert t_{n}-z\right\Vert ^{2}+\left\Vert
x_{n}-z\right\Vert ^{2}-\left\Vert \left( t_{n}-x_{n}\right) +\lambda
_{n}\left( Ax_{n}-Az\right) \right\Vert ^{2}\right] \\
&=&\frac{1}{2}\left[ \left\Vert t_{n}-z\right\Vert ^{2}+\left\Vert
x_{n}-z\right\Vert ^{2}-\left\Vert t_{n}-x_{n}\right\Vert ^{2}\right. \\
&&\left. -2\lambda _{n}\left\langle t_{n}-x_{n},Ax_{n}-Az\right\rangle
-\lambda _{n}^{2}\left\Vert Ax_{n}-Az\right\Vert ^{2}\right] ,
\end{eqnarray*}%
it follows that%
\begin{eqnarray}
\left\Vert t_{n}-z\right\Vert ^{2} &\leq &\left\Vert x_{n}-z\right\Vert
^{2}-\left\Vert t_{n}-x_{n}\right\Vert ^{2}  \notag \\
&&-2\lambda _{n}\left\langle t_{n}-x_{n},Ax_{n}-Az\right\rangle -\lambda
_{n}^{2}\left\Vert Ax_{n}-Az\right\Vert ^{2}.  \label{8}
\end{eqnarray}%
So, using the inequality (\ref{8})\ we get%
\begin{eqnarray*}
\left\Vert x_{n+1}-z\right\Vert ^{2} &=&\left\Vert TP_{C}\left( I-\lambda
_{n}A\right) y_{n}-z\right\Vert ^{2} \\
&=&\left\Vert TP_{C}\left( I-\lambda _{n}A\right) y_{n}-TP_{C}\left(
I-\lambda _{n}A\right) z\right\Vert ^{2} \\
&\leq &\left\Vert y_{n}-z\right\Vert ^{2} \\
&=&\left\Vert \alpha _{n}\left( x_{n}-z\right) +\left( 1-\alpha _{n}\right)
\left( Tt_{n}-z\right) \right\Vert ^{2} \\
&\leq &\alpha _{n}\left\Vert x_{n}-z\right\Vert ^{2}+\left( 1-\alpha
_{n}\right) \left\Vert Tt_{n}-z\right\Vert ^{2} \\
&\leq &\alpha _{n}\left\Vert x_{n}-z\right\Vert ^{2}+\left( 1-\alpha
_{n}\right) \left\Vert t_{n}-z\right\Vert ^{2} \\
&\leq &\left\Vert x_{n}-z\right\Vert ^{2}-\left( 1-\alpha _{n}\right)
\left\Vert t_{n}-x_{n}\right\Vert ^{2} \\
&&-2\lambda _{n}\left( 1-\alpha _{n}\right) \left\langle
t_{n}-x_{n},Ax_{n}-Az\right\rangle \\
&&-\lambda _{n}^{2}\left( 1-\alpha _{n}\right) \left\Vert
Ax_{n}-Az\right\Vert ^{2} \\
&\leq &\left\Vert x_{n}-z\right\Vert ^{2}-\left( 1-d\right) \left\Vert
t_{n}-x_{n}\right\Vert ^{2} \\
&&-2\lambda _{n}\left( 1-\alpha _{n}\right) \left\langle
t_{n}-x_{n},Ax_{n}-Az\right\rangle \\
&&-\lambda _{n}^{2}\left( 1-\alpha _{n}\right) \left\Vert
Ax_{n}-Az\right\Vert ^{2}.
\end{eqnarray*}%
Since $\lim_{n\rightarrow \infty }\left\Vert x_{n+1}-z\right\Vert
=\lim_{n\rightarrow \infty }\left\Vert x_{n}-z\right\Vert $ and $%
Ax_{n}-Az\rightarrow 0,$\ we obtain%
\begin{equation}
\lim_{n\rightarrow \infty }\left\Vert x_{n}-t_{n}\right\Vert =0.  \label{*6}
\end{equation}%
On the other hand, we have%
\begin{eqnarray*}
\left\Vert Tx_{n}-x_{n}\right\Vert &\leq &\left\Vert
Tx_{n}-Tt_{n}\right\Vert +\left\Vert Tt_{n}-x_{n}\right\Vert \\
&\leq &\left\Vert x_{n}-t_{n}\right\Vert +\left\Vert Tt_{n}-x_{n}\right\Vert
.
\end{eqnarray*}%
So, it follows from (\ref{*5}) and (\ref{*6}) that%
\begin{equation}
\lim_{n\rightarrow \infty }\left\Vert Tx_{n}-x_{n}\right\Vert =0.  \label{*7}
\end{equation}

\textbf{Step 4. }Finally, we show that $\left\{ x_{n}\right\} $ converges
weakly to a $p\in F.$ Since $\left\{ x_{n}\right\} $ is a bounded sequence,
there is a subsequence $\left\{ x_{n_{i}}\right\} $ of $\left\{
x_{n}\right\} $ converges weakly to $p.$ We need to show that $p$ belongs to 
$F.$ First, we show that $p\in \Omega .$ From (\ref{*6}), we have $%
t_{n_{i}}\rightharpoonup p.$ Let%
\begin{equation*}
Sv=\left\{ 
\begin{array}{ll}
Av+N_{C}v & ,\text{ }v\in C, \\ 
\emptyset & ,\text{ }v\notin C.%
\end{array}%
\right.
\end{equation*}%
Then $S$ is maximal monotone mapping. Let $\left( v,w\right) \in G\left(
S\right) .$ Since $w-Av\in N_{C}v$ and $t_{n}\in C,$ we get%
\begin{equation}
\left\langle v-t_{n},w-Av\right\rangle \geq 0.  \label{10}
\end{equation}%
On the other hand, from the definiton of $t_{n},$ we have that%
\begin{equation*}
\left\langle x_{n}-\lambda _{n}Ax_{n}-t_{n},t_{n}-v\right\rangle \geq 0
\end{equation*}%
and hence,%
\begin{equation*}
\left\langle v-t_{n},\frac{t_{n}-x_{n}}{\lambda _{n}}+Ax_{n}\right\rangle
\geq 0.
\end{equation*}%
Therefore, using (\ref{10}), we get%
\begin{eqnarray*}
\left\langle v-t_{n_{i}},w\right\rangle &\geq &\left\langle
v-t_{n_{i}},Av\right\rangle \\
&\geq &\left\langle v-t_{n_{i}},Av\right\rangle -\left\langle v-t_{n_{i}},%
\frac{t_{n_{i}}-x_{n_{i}}}{\lambda _{n_{i}}}+Ax_{n_{i}}\right\rangle \\
&=&\left\langle v-t_{n_{i}},Av-Ax_{n_{i}}-\frac{t_{n_{i}}-x_{n_{i}}}{\lambda
_{n_{i}}}\right\rangle \\
&=&\left\langle v-t_{n_{i}},Av-At_{n_{i}}\right\rangle -\left\langle
v-t_{n_{i}},At_{n_{i}}-Ax_{n_{i}}\right\rangle \\
&&-\left\langle v-t_{n_{i}},\frac{t_{n_{i}}-x_{n_{i}}}{\lambda _{n_{i}}}%
\right\rangle \\
&\geq &\left\langle v-t_{n_{i}},At_{n_{i}}-Ax_{n_{i}}\right\rangle
-\left\langle v-t_{n_{i}},\frac{t_{n_{i}}-x_{n_{i}}}{\lambda _{n_{i}}}%
\right\rangle .
\end{eqnarray*}%
Hence, for $i\rightarrow \infty $ we have%
\begin{equation*}
\left\langle v-p,w\right\rangle \geq 0.
\end{equation*}%
Since $S$ is maximal monotone, we have $p\in S^{-1}0$ and hence $p\in \Omega
.$ Next, we show that $p\in F\left( T\right) .$ From (\ref{*7}), Lemma \ref%
{b} and by using $x_{n_{i}}\rightharpoonup p$, we have that $p\in F\left(
T\right) .$ So desired conclusion $\left( p\in F\right) $ is obtained.

Now it remains to show that $\left\{ x_{n}\right\} $ converges weakly to $%
p\in F$ and $p=\lim_{n\rightarrow \infty }P_{F}x_{n}.$ Let assume that there
is an another subsequence $\left\{ x_{n_{j}}\right\} $ of $\left\{
x_{n}\right\} $ and x$_{n_{j}}\rightharpoonup p_{0}\in F.$ We shall show
that $p=p_{0}.$ Conversely, let suppose that $p\neq p_{0}$. By using Opial
condition, we obtain that%
\begin{eqnarray*}
\lim_{n\rightarrow \infty }\left\Vert x_{n}-p\right\Vert
&=&\liminf_{i\rightarrow \infty }\left\Vert x_{n_{i}}-p\right\Vert \\
&<&\liminf_{i\rightarrow \infty }\left\Vert x_{n_{i}}-p_{0}\right\Vert \\
&=&\lim_{n\rightarrow \infty }\left\Vert x_{n}-p_{0}\right\Vert \\
&=&\liminf_{j\rightarrow \infty }\left\Vert x_{n_{j}}-p_{0}\right\Vert \\
&<&\liminf_{j\rightarrow \infty }\left\Vert x_{n_{j}}-p\right\Vert \\
&=&\lim_{n\rightarrow \infty }\left\Vert x_{n}-p\right\Vert .
\end{eqnarray*}%
This is a contradiction, so we get $p=p_{0}.$ This implies that $%
x_{n}\rightharpoonup p\in F.$

Finally, we need to show $p=\lim_{n\rightarrow \infty }P_{F}x_{n}.$ Since $%
p\in F,$ we have%
\begin{equation*}
\left\langle p-P_{F}x_{n},P_{F}x_{n}-x_{n}\right\rangle \geq 0.
\end{equation*}%
By Lemma \ref{c}, $\left\{ P_{F}x_{n}\right\} $ converges strongly to $%
u_{0}\in F.$ Then, we get%
\begin{equation*}
\left\langle p-u_{0},u_{0}-p\right\rangle \geq 0,
\end{equation*}%
and hence $p=u_{0}.$ So, proof is completed.
\end{proof}

\begin{corollary}
Let $C$ be a nonempty closed convex subset of a real Hilbert space $H.$ Let $%
A:C\rightarrow H$ be an $\alpha $-inverse strongly monotone mapping such
that $\Omega \neq \emptyset .$ For arbitrary initial value $x_{0}\in H,$ let 
$\left\{ x_{n}\right\} $ be a sequence defined by%
\begin{equation*}
\left\{ 
\begin{array}{l}
x_{n+1}=P_{C}\left( I-\lambda _{n}A\right) y_{n} \\ 
y_{n}=\alpha _{n}x_{n}+\left( 1-\alpha _{n}\right) P_{C}\left( I-\lambda
_{n}A\right) x_{n},\forall n\geq 0,%
\end{array}%
\right.
\end{equation*}%
where $\{\lambda _{n}\}\subset \lbrack a,b]$ for some $a,b\in (0,2\alpha )$
and $\left\{ \alpha _{n}\right\} \subset \left[ c,d\right] $ for some $%
c,d\in \left( 0,1\right) $. Then, the sequence $\left\{ x_{n}\right\} $
converges weakly to a point $p\in \Omega $ where $p=\lim_{n\rightarrow
\infty }P_{\Omega }x_{n}.$
\end{corollary}

\section{Applications}

Let $B:H\rightarrow 2^{H}$ be a maximal monotone mapping. The resolvent of $%
B $ of order $r>0$ is the single valued mapping $J_{r}^{B}:H\rightarrow H$
defined by%
\begin{equation*}
J_{r}^{B}x=\left( I+rB\right) ^{-1}x
\end{equation*}%
for any $x\in H$. It is easy to check that $F\left( J_{r}^{B}\right)
=B^{-1}0 $. Moreover, the resolvent $J_{r}^{B}$ is a nonexpansive mapping.
So, we can give the following theorem.

\begin{theorem}
Let $H$ be a real Hilbert space$.$ Let $\alpha >0$, $A:H\rightarrow H$ be an 
$\alpha $-inverse strongly monotone mapping and $B:H\rightarrow 2^{H}$ be a
maximal monotone mapping such that $A^{-1}0\cap B^{-1}0\neq \emptyset .$ For
arbitrary initial value $x_{0}\in H,$ let $\left\{ x_{n}\right\} $ be a
sequence defined by%
\begin{equation*}
\left\{ 
\begin{array}{l}
x_{n+1}=J_{r}^{B}\left( y_{n}-\lambda _{n}Ay_{n}\right) \\ 
y_{n}=\alpha _{n}x_{n}+\left( 1-\alpha _{n}\right) J_{r}^{B}\left(
x_{n}-\lambda _{n}Ax_{n}\right) ,\forall n\geq 0,%
\end{array}%
\right.
\end{equation*}%
where $\{\lambda _{n}\}\subset \lbrack a,b]$ for some $a,b\in (0,2\alpha )$
and $\left\{ \alpha _{n}\right\} \subset \left[ c,d\right] $ for some $%
c,d\in \left( 0,1\right) $. Then the sequence $\left\{ x_{n}\right\} $
converges weakly to a point $p\in A^{-1}0\cap B^{-1}0$ where $%
p=\lim_{n\rightarrow \infty }P_{A^{-1}0\cap B^{-1}0}x_{n}.$
\end{theorem}

\begin{proof}
We have $A^{-1}0=VI\left( H,A\right) $, $F\left( J_{r}^{B}\right) =B^{-1}0$
and $P_{H}=I$. Since the resolvent $J_{r}^{B}$ is a nonexpansive mapping, we
obtain the desired conclusion.
\end{proof}

Now, we give a theorem for a pair of nonexpansive mapping and strictly
pseudocontractive mapping. A mapping $S:C\rightarrow C$ is called $k$-
strictly pseudocontractive mapping if there exists $k$ with $0\leq k<1$ such
that%
\begin{equation*}
\left\Vert Sx-Sy\right\Vert ^{2}\leq \left\Vert x-y\right\Vert
^{2}+k\left\Vert \left( I-S\right) x-\left( I-S\right) y\right\Vert ^{2}
\end{equation*}%
for all $x,y\in C.$ Let $A=I-S.$ Then, it is known that the mapping $A$ is
inverse strongly monotone mapping with $\left( 1-k\right) /2$, i.e.,%
\begin{equation*}
\left\langle Ax-Ay,x-y\right\rangle \geq \frac{1-k}{2}\left\Vert
Ax-Ay\right\Vert ^{2}.
\end{equation*}

\begin{theorem}
Let $C$ be a nonempty closed convex subset of a real Hilbert space $H.$ Let $%
T:C\rightarrow C$ be a nonexpansive mapping and $S:C\rightarrow C$ be a $k$-
strictly pseudocontractive mapping such that $F\left( T\right) \cap F\left(
S\right) \neq \emptyset .$ For arbitrary initial value $x_{0}\in H,$ let $%
\left\{ x_{n}\right\} $ be a sequence defined by%
\begin{equation*}
\left\{ 
\begin{array}{l}
x_{n+1}=T\left( \left( I-\lambda _{n}\right) y_{n}+\lambda _{n}Sy_{n}\right)
\\ 
y_{n}=\alpha _{n}x_{n}+\left( 1-\alpha _{n}\right) T\left( \left( I-\lambda
_{n}\right) x_{n}+\lambda _{n}Sx_{n}\right) ,\forall n\geq 0,%
\end{array}%
\right.
\end{equation*}%
where $\{\lambda _{n}\}\subset \lbrack a,b]$ for some $a,b\in (0,1-k)$ and $%
\left\{ \alpha _{n}\right\} \subset \left[ c,d\right] $ for some $c,d\in
\left( 0,1\right) $. Then the sequence $\left\{ x_{n}\right\} $ converges
weakly to a point $p\in F\left( T\right) \cap F\left( S\right) $ where $%
p=\lim_{n\rightarrow \infty }P_{F\left( T\right) \cap F\left( S\right)
}x_{n}.$
\end{theorem}

\begin{proof}
Let $A=I-S.$ Then, we know that $A$ is inverse strongly monotone mapping.
Also, It is clear that $F\left( S\right) =VI\left( C,A\right) .$ Since, $A$
is a mapping from $C$ into itself, we get%
\begin{equation*}
\left( I-\lambda _{n}\right) x_{n}+\lambda _{n}Sx_{n}=x_{n}-\lambda
_{n}\left( I-S\right) x_{n}=P_{C}\left( x_{n}-\lambda _{n}Ax_{n}\right) .
\end{equation*}%
So, from Theorem \ref{11}, we obtain the desired conclusion.
\end{proof}

\begin{theorem}
Let $H$ be a real Hilbert space$.$ Let $\alpha >0$, $A:H\rightarrow H$ be an 
$\alpha $-inverse strongly monotone mapping and $T:H\rightarrow H$ be a
nonexpansive mapping such that $F\left( T\right) \cap A^{-1}0\neq \emptyset
. $ For arbitrary initial value $x_{0}\in H,$ let $\left\{ x_{n}\right\} $
be a sequence defined by%
\begin{equation*}
\left\{ 
\begin{array}{l}
x_{n+1}=T\left( y_{n}-\lambda _{n}Ay_{n}\right) \\ 
y_{n}=\alpha _{n}x_{n}+\left( 1-\alpha _{n}\right) T\left( x_{n}-\lambda
_{n}Ax_{n}\right) ,\forall n\geq 0,%
\end{array}%
\right.
\end{equation*}%
where $\{\lambda _{n}\}\subset \lbrack a,b]$ for some $a,b\in (0,2\alpha )$
and $\left\{ \alpha _{n}\right\} \subset \left[ c,d\right] $ for some $%
c,d\in \left( 0,1\right) $. Then the sequence $\left\{ x_{n}\right\} $
converges weakly to a point $p\in VI\left( F\left( T\right) ,A\right) $
where $p=\lim_{n\rightarrow \infty }P_{F\left( T\right) \cap A^{-1}0}x_{n}.$
\end{theorem}

\begin{proof}
We have $A^{-1}0=VI\left( H,A\right) $ and $P_{H}=I.$ Also, it is clear that 
$F\left( S\right) \cap A^{-1}0\subset VI\left( F\left( S\right) ,A\right) $.
So, by Theorem \ref{11}, we get the desired conclusion.
\end{proof}

\end{document}